\documentclass[pdflatex,sn-mathphys-num]{sn-jnl}% Math and Physical Sciences Numbered Reference Style
\usepackage{graphicx}%
\usepackage{multirow}%
\usepackage{amsmath,amssymb,amsfonts}%
\usepackage{amsthm}%
\usepackage{mathrsfs}%
\usepackage[title]{appendix}%
\usepackage{xcolor}%
\usepackage{textcomp}%
\usepackage{manyfoot}%
\usepackage{booktabs}%
\usepackage{algorithm}%
\usepackage{algorithmicx}%
\usepackage{algpseudocode}%
\usepackage{listings}%
\usepackage{upgreek}
\usepackage{tabularx}
\usepackage{array}

\theoremstyle{thmstyleone}%
\newtheorem{theorem}{Theorem}%  meant for continuous numbers
\newtheorem{proposition}[theorem]{Proposition}% 

\theoremstyle{thmstyletwo}%

\theoremstyle{thmstylethree}%
\newtheorem{definition}{Definition}%

\begin{document}

\title[Article Title]{\centering A Data-embedded Solution Paradigm for Nonconvex Probable Event Constrained Optimization}


\author{\fnm{Qifeng} \sur{Li}}\email{qifeng.li@ucf.edu}



\affil{\orgdiv{Department of Electrical and Computer Engineering}, \orgname{University of Central Florida}, \orgaddress{\street{4000 Central Florida Blvd}, \city{Orlando}, \postcode{32816}, \state{Florida}, \country{USA}}}

% \affil[2]{\orgdiv{Department}, \orgname{Organization}, \orgaddress{\street{Street}, \city{City}, \postcode{10587}, \state{State}, \country{Country}}}

% \affil[3]{\orgdiv{Department}, \orgname{Organization}, \orgaddress{\street{Street}, \city{City}, \postcode{610101}, \state{State}, \country{Country}}}


\abstract{This paper introduces a new modeling framework for optimization under uncertainty, called Probable Event Constrained Optimization (PECO). Unlike conventional chance-constrained formulations, which only limit the probability of constraint violation, PECO also explicitly requires feasibility for all events whose probability exceeds a prescribed threshold. This guarantees that solutions remain valid across all high-probability realizations of uncertainty. To solve PECO, we proposed a data-embedded program (DEP) which directly incorporates historical measurements of the uncertain parameters to obtain a deterministic approximation for PECO. While existing solution methods for optimization problems under uncertainty rely heavily on convexity or linearity assumptions, the proposed data-embedded solution paradigm provides a unique opportunity for solving nonlinear and nonconvex PECOs. The effectiveness of this approach depends on properly estimating the number of elements in the family of solution-determining data sets. As we enter the era of big data, this information can be properly estimated by leveraging the power of machine learning. }




\maketitle

\section{Introduction}\label{sec1}

Optimization under uncertainty is central to modern operations research, with applications ranging from finance and logistics to engineering. Among established approaches, robust optimization (RO) \cite{ben2009robust} secures solutions that are feasible to all possible realizations of uncertainty, while distributionally robust optimization (DRO) \cite{rahimian2019distributionally} extends this idea to all possible probability distributions. Chance-constrained optimization (CCO) \cite{li2008chance} allows constraint violations with small probability. Although these approaches have proven effective, they share an important limitation: they guarantee feasibility in a probabilistic aggregate sense, rather than at the level of every high-probability event. As a result, certain highly likely scenarios may still be infeasible—even if the overall violation probability remains below the confidence level of CCO. This gap is critical in engineering applications where decisions are required to be optimal for high-probability events to maintain high efficiency. In contrast, for low-probability but high-impact events, optimality is typically not required; instead, safety becomes the primary concern when such events occur. To address this issue, we introduce Probable Event Constrained Optimization (PECO) which enforces not only that the total feasibility probability but also that all events whose probability exceeds a user-defined threshold are themselves fully feasible. In this sense, PECO strengthens the probabilistic guarantee from a global measure to an event-wise requirement. The differences between PECO and the existing models are further compared in Table \ref{tb:comparison}.
\begin{table}[h]
\footnotesize
\centering
\caption{Comparison of modeling and solution paradigms}
\vspace{-6pt}
\begin{tabular}{|m{0.88cm}|m{3.26cm}|m{2.88cm}|m{4.28cm}|}
\hline
\textbf{Model}   & \textbf{Logical   Guarantee}                             & \textbf{Solution   method}          & \textbf{Key   limitations/Advantages}                            \\ \hline
RO               & Feasibility for the whole uncertainty set                                      & Min-max approximation \cite{sniedovich2008wald}                                   & Ignores probability information; overly   conservative; strong convexity requirements \\ \hline
CCO              & Bounded probability of constraint violation                                    & Stochastic/robust approximations \cite{han2016robust}, and scenario   methods \cite{calafiore2005uncertain} & Does not ensure feasibility for high-probability   events; convexity-dependent        \\ \hline
DRO              & Feasibility for all distributions in an ambiguity   set                        & Dual approximation of inner supremum \cite{delage2010distributionally}                    & Requires carefully designed ambiguity sets; heavy   reliance on convexity             \\ \hline
PECO (this work) & Feasibility to all high-probability realizations under   unknown distributions & DESP (this work)                                                     & No explicit uncertainty or ambiguity sets; robust to   nonlinearity and nonconvexity  \\ \hline
\end{tabular}
\label{tb:comparison}
\end{table}

Another critical issue associated with existing research is the lack of effective solution methods for nonlinear and nonconvex optimization problems under uncertainty. Despite the fact that engineering problems are in general nonconvex, existing mature solution methods, as listed in Table \ref{tb:comparison}, rely on rigorous assumptions of convexity. To address this issue, we developed a data-embedded solution paradigm (DESP) which provides a unique opportunity for solving nonlinear, nonconvex PECOs. Under DESP, a data-embedded program (DEP) is proposed as the deterministic approximation of PECO. The DEP is constructed by directly embedding historical measurements of the uncertain parameters into the optimization problem. Without relying on assumptions of convexity, a probabilistic metric was developed to evaluate the accuracy of approximating a PECO with the DEP model based on an estimate of the number of elements in the family of solution-determining data sets (SDDS\footnote{See Definition \ref{def:rdsf}}). Although a valid upper or lower bound for this quantity may not exist in general, for a specific engineering problem it can often be predicted with satisfactory accuracy using state-of-the-art machine learning methods.

To sum up, the contributions of this work are twofold: 1) we propose a new optimization model--PECO--that guarantees both the low probability of constraint violation and feasibility to high-probability events; and 2) we develop a data-embedded solution paradigm (DESP) that enables the solution of nonconvex PECOs. The remainder of this paper is organized as follows. Section 2 introduces the key properties of the PECO model and positions PECO relative to RO, CCO, RO-CCO, and DRO. Section 3 presents the DEP deterministic approximation. Section 4 summarizes the overall procedure of DESP, while future work is discussed in Section 5.
% \begin{enumerate} 
% %\vspace{-6pt}
% \itemsep0em 
%     \item A new optimization model--PECO--that strengthens CCO by ensuring feasibility for all high-probability events, offering reliability beyond distributionally robust or chance-constrained formulations.
%     \item A data-embedded deterministic approximation (DeDA) that transforms PECO into a deterministic form by directly embedding observed data points, bypassing minimax formulations.
%     \item Strategic data-selection (SDS): an efficient procedure to identify only those data points that define feasibility boundaries, thereby improving scalability.
%     \item Applications to nonlinear, nonconvex problems: unlike many robust and chance-constrained methods, PECO with DeDA is suitable for general optimization settings.
% \end{enumerate}

\section{Probable Event Constrained Optimization (PECO)}\label{sec2}
This section introduces the mathematical formulation of PECO and highlights several of its key properties by comparing it with existing modeling frameworks for optimization under uncertainty, such as CCO, the robust approximation of CCO, and DRO.

\subsection{Formulations of PECO}
\subsubsection{A general formulations}
Let $\mathcal{E}$ denote the set of all events and let $\mathbb{P}[E]$ denote the probability of event $E$, we have the following definition.
\begin{definition}[Probable event] \label{def:pe}
    An event $E \in \mathcal{E}$ is called a probable event if $\mathbb{P}[E]\ge \alpha$, where $\alpha$ is a user-defined probability threshold and the probability distribution of $E$ may be \textit{unknown}\footnote{Given that the literature often uses $\mathbb{P}_P[E]$ to denote the probability of event $E$ under probability distribution $P$, we remove the subscript $P$ to indicate that the proposed PECO modeling paradigm does not require knowledge of $P$.}. 
\end{definition}
The general PECO model is then
\begin{subequations}  \label{PECO1}
\begin{align} 
 \min_{x\in \mathbb{R}^n} \; & f(x)  \label{obj_PECO} \\
 \mathrm{s.t.} \;  & g(x,y(E)) \le 0,\;\begin{cases} 
 &\forall E \in \mathcal{E}_\alpha = \{ E \in \mathcal{E} \,|\,\mathbb{P}[E] \ge \alpha,\,\text{and the} \\& \text{probability distribution of}\,E\, \text{may be unknown}.\} \end{cases} \label{con_PECO}
\end{align} 
\end{subequations}
where $x$ is the decision vector, $y$ represents the vector of parameters that are associated with an uncertain event $E$, $f:\mathbb{R}^n \rightarrow \mathbb{R}$, and $g:\mathbb{R}^n \times \Upxi \rightarrow \mathbb{R}^m$ are smooth functions. This formulation enforces feasibility for all probable events—a requirement stronger than the probabilistic guarantee in CCO. Constraint (\ref{con_PECO}) is  referred to as Probable Event Constraint (PEC).

\subsubsection{A special case}
This paper focuses on a special case of the above PECO formulation, where events correspond to realizations of a random vector $\xi$, i.e., $E:y=\xi$. Such uncertain events are common in engineering systems. For example, in power systems, uncertain outputs of renewable generators define such events \cite{li2021uncertainty}. When $\mathbb{P}[\xi]$ is used to denote $\mathbb{P}[y=\xi]$ for simplicity, model (\ref{PECO1}) becomes
\begin{subequations} \label{PECO2}
\begin{align} 
\text{\textbf{PECO}:}\; \min_{x\in \mathbb{R}^n} \; & (\text{\ref{obj_PECO}})  \nonumber \\
 \mathrm{s.t.} \;  & g(x,\xi) \le 0,\;\forall \xi \in \Upxi_\alpha = \{ \xi \in \Upxi \,|\,\mathbb{P}[\xi] \ge \alpha\} \label{PEC} 
\end{align}
\end{subequations} %,\,(P\in \mathcal{P} \, \text{or is unknown})
where $\Upxi \subset (\mathbb{Z}^{u_1},\mathbb{R}^{u_2})$ ($u_1+u_2=u$) and the probability distribution of event $y=\xi$ may be \textit{unknown}. The PECO model then ensures feasibility across all realizations of $\xi$ whose probability exceeds $\alpha$. Denoting the feasible set of PECO in the $x$-space as:
\begin{equation} \label{feasisetpccsp}
    \mathcal{X}_{\rm P}(\alpha)=\{ x \in \mathbb{R}^n \mid g(x,\xi) \le 0,\;\forall \xi \in \Upxi_\alpha\},
\end{equation}
we have the following proposition.
\begin{proposition} \label{pro:feasisetpccsp}
    $\mathcal{X}_{\rm P}(\alpha_1) \supseteq \mathcal{X}_{\rm P}(\alpha_2)$ if $\alpha_1 \ge \alpha_2$.
\end{proposition}
\textit{Proof}: When $\alpha_1 \ge \alpha_2$, for an arbitrary outcome $\check{\xi}_{\rm s}$ of $\xi$, $\mathbb{P}[\xi=\check{\xi}_{\rm s}] \ge \alpha_2$ if $\mathbb{P}[\xi=\check{\xi}_{\rm s}] \ge \alpha_1$. In other words,  $\Upxi_{\alpha_1} \subseteq \Upxi_{\alpha_2}$. Then,
\begin{align}
    \mathcal{X}_{\rm P}(\alpha_2)&=\left\{x \in \mathbb{R}^n \middle\vert \begin{array}{l}
        g(x,y) \le 0,\;\forall y\in \Upxi_{\alpha_1} \\
        g(x,y) \le 0,\;\forall y\in \Upxi_{\alpha_2} \setminus \Upxi_{\alpha_1}
    \end{array} \right\} \nonumber \\
    &=\{x \in \mathcal{X}_{\rm P}(\alpha_1) \mid g(x,y) \le 0,\;\forall y\in \Upxi_{\alpha_2} \setminus \Upxi_{\alpha_1} \} \nonumber
  \\
    & \subseteq \mathcal{X}_{\rm P}(\alpha_1). \nonumber
\end{align}
 \hfill $\Box$

\subsection{Relation to Existing Modeling Frameworks}
The logical guarantees of the proposed PECO model and the existing models, such as RO, CCO, and DRO, were compared in the second column of Table 1. This subsection aims to reveal more differences between PECO and these existing models.
\subsubsection{Compared to RO}
PECO, whose key element is the PEC as given in (\ref{PEC}), has a different logic meaning from the classic RO as explained in Section 1 and Table \ref{tb:comparison}. While there are many variants of RO \cite{ben2000robust,bertsimas2004price,wang2016likelihood,ben2013robust}, a standard formulation is given as
\begin{equation}
    \text{\text{Classic RO}:}\quad\min_{x} \;  (\text{\ref{obj_PECO}}) \quad \mathrm{s.t.}\; g(x,\xi) \le 0, \; \forall \xi \in \Upxi. \label{RO}
\end{equation}
Although PECO and RO share a similar mathematical structure, an important difference lies between $\Upxi$ and $\Upxi_\alpha$. Specifically, 
\begin{enumerate}
    \item $\Upxi$ is a deterministic set, whereas $\Upxi_\alpha$ is probabilistic.
    \item $\Upxi_\alpha$ may be nonconvex and even disconnected, even when $\Upxi$ is convex and the distribution is simple.
    \item The geometric properties of $\Upxi_\alpha$ can change as $\alpha$ varies.
\end{enumerate}
An example of $\Upxi_\alpha$ is provided in Figure \ref{fig:contour}, where $\xi_1$ and $\xi_2$ follow bimodal distributions. To be specific, $\mathbb{P}[\xi_1] \sim [0.5\mathcal{N}(-2,1)+0.5\mathcal{N}(3,1.8)]$ and $\mathbb{P}[\xi_2] \sim [0.5\mathcal{N}(0,1.2)+0.5\mathcal{N}(5,1.6)]$, where $\mathcal{N}(\mu,\sigma^2)$ is a normal distribution with $\mu$ as the mean and $\sigma^2$ as the variance. Although the mathematical expression of $\Upxi_\alpha$ can be derived explicitly for this example (see (\ref{upxialpha})), it remains highly nonconvex and may even become disconnected as $\alpha$ changes. This example illustrates that, in addition to the inherent nonconvexity of system constraints, the set $\Upxi_\alpha$ itself can exhibit severe nonconvexity. Moreover, constructing $\Upxi_\alpha$ is often nontrivial in real-world applications, regardless of whether it is convex or nonconvex.
 \begin{equation}\label{upxialpha}
 \begin{aligned}
    \Upxi_\alpha=&\{(\xi_1,\xi_2)\in \Upxi\,|\, (\frac{0.5}{\sqrt{2\pi}}e^{-(\xi_1+2)^2} +\frac{0.5}{\sqrt{3.6\pi}}e^{-\frac{(\xi_1-3)^2}{1.8}})\\
     &(\frac{0.5}{\sqrt{2.4\pi}}e^{-\frac{\xi_2^2}{1.2}}+ \frac{0.5}{\sqrt{3.2\pi}}e^{-\frac{(\xi_2-5)^2}{1.6}})\ge \alpha \}. 
\end{aligned}
 \end{equation}

\begin{figure}[h]
    \centering
    \includegraphics[width=0.388\linewidth]{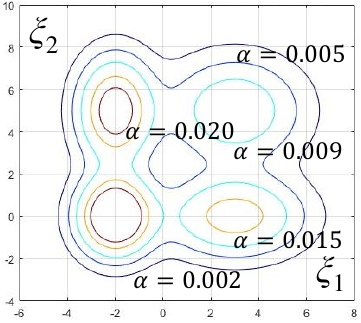}
    \caption{An illustrative example of $\Upxi_\alpha$}
    \label{fig:contour}
\end{figure}
\subsubsection{Compared to CCO}
 While CCO has various forms \cite{pagnoncelli2009sample,geletu2013advances,gopalakrishnan2021solving}, our discussion is based on the following standard formulation:
 \begin{equation}
      \text{\textbf{Classic CCO}:}\quad\min_{x\in \mathbb{R}^n} \;  (\text{\ref{obj_PECO}}) \quad\quad \mathrm{s.t.}\; \mathbb{P}[g(x,\xi) \le 0] \ge 1-\beta .\label{CCO}
 \end{equation}
The logical meanings of CCO and PECO can be compared as follows:
\begin{itemize}
    \item The chance constraint in (\ref{CCO}) requires that the constraint is satisfied with high probability (i.e., $\ge 1- \beta$).
    \item PEC (\ref{PEC}) enforces that the constraint must hold for all probable realizations, i.e., every realization that occurs with probability at least $\alpha$.
\end{itemize}
If the probability distribution $P(\xi)$ is known, a deterministic formulation of the feasible space of CCO (\ref{CCO}) in the $x$-space is given as 
\begin{equation} \label{feasisetcccsp1}
    \mathcal{X}_{\rm C}=\left\{ x \in \mathbb{R}^n \middle| 
\int\cdots\int_{ \mathcal{M}(x)}P(\xi)d\xi_u\cdots d\xi_1\ge 1-\beta \right\},
\end{equation}
where $$\mathcal{M}(x):=\{\xi \in \Upxi|g(x,\xi) \le 0, x \in \mathbb{R}^n\}$$ denotes the set of realizations of $\xi$ for which the constraint is satisfied for a given $x$. We have the following propositions. 
\begin{proposition}[on the relations between $\mathcal{X}_{\rm P}$ and $\mathcal{X}_{\rm C}$] \label{pro:CD}
     If
    \begin{equation} \label{CD}
  \alpha = \arg_v \left\{ \int\cdots\int_{\{\xi \in \Upxi|P(\xi)\ge v\}}P(\xi)d\xi_u\cdots d\xi_1=1-\beta \right\} ,
\end{equation}
where $\arg$ means the argument of a function, we have the following relations:
\begin{enumerate}
    \item  $\mathcal{X}_{\rm P} \subseteq \mathcal{X}_{\rm C}$;
    \item  $\mathcal{X}_{\rm P} = \mathcal{X}_{\rm C}$ if, for arbitrary realizations $\xi^{(a)} \in \mathcal{M}(x)$ and $\xi^{(b)} \in \Upxi \setminus \mathcal{M}(x)$, it holds that $P(\xi^{(a)}) \ge P(\xi^{(b)})$.
\end{enumerate}
\end{proposition}
\textit{Proof}: Consider the experiment in which a realization of $\xi$ is randomly drawn from the sample space $\Upxi$, and let $\check{\xi}$ denote this outcome. According to probability theory, the probability of event $\mathbb{P}[\xi=\check{\xi}]\ge \alpha$, i.e., $\mathbb{P}[\mathbb{P}[\xi=\check{\xi}]\ge \alpha]$, is given by 
\begin{equation}
    \mathbb{P}[\mathbb{P}[\xi=\check{\xi}]\ge \alpha]=\mathbb{P}[P(\check{\xi})\ge \alpha]=\int\cdots\int_{\{\xi \in \Upxi|P(\xi)\ge \alpha\}}P(\xi)d\xi_u\cdots d\xi_1.
\end{equation}
 Therefore, condition (\ref{CD}) implies that $$\mathbb{P}[P(\check{\xi})\ge \alpha]=1-\beta.$$ Let $x_{\rm P} \in \mathcal{X}_{\rm P}$ denote an arbitrarily feasible solution of PECO, the logical meaning of PEC indicates that $g(x_{\rm P},\check{\xi}) \le 0$ whenever $\mathbb{P}[\xi=\check{\xi}]\ge \alpha$. Consequently, the probability of constraint satisfaction is not less than the probability of event $\mathbb{P}[\xi=\check{\xi}]\ge \alpha$, i.e., $$\mathbb{P}[g(x_{\rm P},\check{\xi}) \le 0] \ge \mathbb{P}[P(\check{\xi})\ge \alpha]=1-\beta.$$ Therefore, according to the logical meaning of the chance constraint, we have $x_{\rm P} \in \mathcal{X}_{\rm C}$ which implies that $\mathcal{X}_{\rm P} \subseteq \mathcal{X}_{\rm C}$.

For the second statement, the condition ``$P(\xi^{(a)}) \ge P(\xi^{(b)})$ for arbitrary realizations $\xi^{(a)} \in \mathcal{M}(x)$ and $\xi^{(b)} \in \Upxi \setminus \mathcal{M}(x)$" implies that there exists a threshold $v$ such that $P(\xi^{(a)}) \ge v$ and $P(\xi^{(b)}) \le v$. This implies   
\begin{equation} \label{proofCD}
    P(\xi)\ge v,\; \forall \xi \in \mathcal{M}(x)
\end{equation}
It's not hard to know that $v = \alpha$ in (\ref{proofCD}) under condition (\ref{CD}). Let $\hat{\xi}$ be an arbitrary realization of $\xi$ in $\mathcal{M}(x)$, we have $\mathbb{P}[\xi=\hat{\xi}]\ge \alpha$. Let $x_{\rm C}$ denote an arbitrarily feasible solution of CCO satisfying the above condition, we have $g(x_{\rm C},\hat{\xi}_{\rm s}) \le 0$, which implies that $x_{\rm C} \in \mathcal{X}_{\rm P}$ and, namely, $\mathcal{X}_{\rm P} \supseteq \mathcal{X}_{\rm C}$. Therefore, we have $\mathcal{X}_{\rm P} = \mathcal{X}_{\rm C}$.

\hfill $\Box$

Proposition \ref{pro:CD} indicates that PECO (\ref{PECO2}) can be viewed as a \emph{safe} alternative to CCO (\ref{CCO}) in terms of feasible space. To be specific, although the classic chance constraint ensures that feasible solutions achieve a satisfactory probability of constraint satisfaction, it cannot guarantee that the optimal solution is feasible for a probable realization of $\xi$, which is often undesirable in engineering applications. An advantage of PECO is that it guarantees both properties.

\subsubsection{Compared to the robust approximation of CCO}
When the functions $f$ and $g$ are convex, some existing work approximates the classic CCO (\ref{CCO}) by the following RO problem \cite{nemirovski2007convex,nemirovski2009robust}:
\begin{equation}
    \text{\textbf{RO-CCO}:}\quad\min_{x\in \mathbb{R}^n} \;  f(x) \quad\quad \mathrm{s.t.}\; g(x,\xi) \le 0, \; \forall \xi \in \mathcal{U}\subset \Upxi \label{ROCCO}
\end{equation}
where $\mathcal{U}$ is convex such that the robust constraint $g(x,\xi) \le 0 \; (\forall \xi \in \mathcal{U})$ is an inner approximation of the chance constraint $\mathbb{P}[g(x,\xi) \le 0] \ge 1-\beta$. The proposed PECO modeling paradigm differs from the RO-CCO method in various aspects. First, $\Upxi_\alpha$ has a fundamentally different logic meaning from $\mathcal{U}$ and is not necessarily convex. Second, existing research on RO-CCO typically does not consider nonconvex $\mathcal{U}$ since convexity of $\mathcal{U}$ is essential for maintaining tractability. 

Strictly speaking, RO-CCO is not a modeling paradigm for optimization under uncertainty like RO, CCO, DRO, and PECO. Instead, It is just a solution method of CCO which leverages the computational tractability of the RO formulation under the condition of strict convexity. Consequently, existing research on RO-CCO primarily focuses on constructing a convex set $\mathcal{U}$ such that $g(x,\xi) \le 0 \; (\forall \xi \in \mathcal{U})$ is a tight inner approximation of the chance constraint $\mathbb{P}[g(x,\xi) \le 0] \ge 1-\beta$.

\subsubsection{Compared to DRO}
  While DRO has many variants \cite{delage2010distributionally,mohajerin2018data,hanasusanto2015distributionally,calafiore2006distributionally,hu2013kullback}, a standard formulation is given as
 \begin{equation} \label{DRO}
     \text{DRO:}\quad\min_{x} \;  \max_{\forall P \in \mathcal{P}}\mathbb{E}[f(x,\xi)] \quad \mathrm{s.t.}\; g(x) \le 0,
 \end{equation}
 where $\mathcal{P}$ represents the ambiguity set.
 DROs can generally be classified into two categories: distributionally robust stochastic optimization (DRSO) and distributionally robust chance-constrained optimization (DRCCO). A central focus of existing research on DRO lies in constructing a suitable ambiguity set with the following objectives: 1) the uncertainty can be properly captured, and 2) the resulting deterministic approximation is computationally tractable. 
 
 DRO has recently attracted significant attention because it seeks robustness with respect to a family of probability distributions rather than a single distribution. However, in practice, DRO can be inconvenient for real-world applications. First, constructing the ambiguity set still requires distributional information and assumptions—such as moments or support—which, if accurately known, effectively specify the underlying distribution. Second, the resulting bilevel formulations are computationally demanding and typically require strong convexity assumptions to be tractable. In contrast, PECO does not rely on this information thanks to the DESP. Solutions to DRO problems typically rely on mathematical approximation, such as dual reformulation, while the DESP of PECO (described in the next section) is completely based on data. 

 \subsection{Summary}
PECO offers important advantages, while also introducing significant challenges. On the one hand, PECO is logically well aligned with engineering applications, such as power system optimization, as it guarantees feasibility for all high-probability realizations under unknown distributions. On the other hand, PECO poses substantial computational challenges when addressed using existing solution methods, which are not designed to handle the complexity of~$\Upxi_\alpha$. These challenges motivate the development of a novel solution paradigm—DESP—for the efficient solution of PECO, as described in the next section.

\section{Data-embedded Programming}\label{sec3}
When the mathematical expression of $\Upxi_\alpha$ is available, and both $g$ and $\Upxi_\alpha$ are convex, Wald's minimax method can be applied to reformulate PECO (\ref{PECO2}) as a single-level deterministic program. However, since the minimax method was originally designed for linear problems, it may become computationally intractable or even fail to produce meaningful solutions for nonlinear and nonconvex problems, as it is fundamentally based on worst-case satisfaction. Meanwhile, as we enter the era of big data, large volumes of data are increasingly available. Motivated by this, we propose a DESP for PECO to overcome the limitations of the minimax approach. In this paradigm, a DEP model is introduced as a deterministic reformulation of PECO.

\subsection{Mathematical Formulation}
\subsubsection{Preliminary}
To better reveal the properties of DEP, several key terminologies are defined in Table \ref{tab:terminologies}. An illustrative example is provided in Figure \ref{fig:illu_examp1}, where the uncertain vector $\xi =[\xi_1,\xi_2]^{\rm T}$ and the uncertainty set is given by $$\Upxi=\{\xi  \in \mathbb{Z}^2 | \xi_1^2+\xi_2^2-4\xi_1-4\xi_2\le -5.9\}.$$ This example considers 100 data points that are measured in history, forming the data set $\mathcal{D}=\{\xi^{(k)}_{\rm d},\; k=1,\ldots,100\}$. Suppose the scenarios are ordered as $\xi^{(1)}_{\rm s}=(1,1)$, $\xi^{(2)}_{\rm s}=(1,2)$, $\xi^{(3)}_{\rm s}=(1,3)$, $\xi^{(4)}_{\rm s}=(2,1)$, $\xi^{(5)}_{\rm s}=(2,2)$, $\xi^{(6)}_{\rm s}=(2,3)$, $\xi^{(7)}_{\rm s}=(3,1)$, $\xi^{(8)}_{\rm s}=(3,2)$, and $\xi^{(9)}_{\rm s}=(3,3)$, the set of all scenarios is $\mathcal{S}^\forall=\{\xi^{(k)}_{\rm s},\; k=1,\ldots,9\}$ which is a finite set since $\xi_1$ and $\xi_2$ are two finite discrete parameters. From Figure \ref{fig:illu_examp1}, scenario $\xi^{(2)}_{\rm s}$ occurred 9 times in the data set $\mathcal{D}$, i.e., $\xi^{(j)}_{\rm d}=\xi^{(2)}_{\rm s}$ for $j=1,\ldots,9$. Its empirical probability is therefore $\mathbb{P}[\xi=\xi^{(2)}_{\rm s}]=9/100=0.09$, and probabilities for other scenarios can be computed similarly. If $\alpha=0.1$, the set of probable scenarios $\mathcal{S}^\forall_\alpha=\{\xi^{(k)}_{\rm s},\; k=4, 5, 6, 8\}$. For this example, the explicit mathematical expression of $\Upxi_\alpha$ may not exist. The set $\mathcal{D}_\alpha$ contains all data points equal to $\xi^{(k)}_{\rm s}$ ($k=4, 5, 6, 8$) and $|\mathcal{D}_\alpha|=85$. It follows that $\mathcal{S}^\forall_\alpha=\mathfrak{U}[\mathcal{D}_\alpha]$ where $\mathfrak{U}[\cdot]$ denotes the underlying set of a multiset\footnote{A underlying set is the set of distinct elements of a multiset.}.
\begin{table}[h]
\footnotesize
\centering
\caption{Key terminologies.}
\vspace{-6pt}
\label{tab:terminologies}
\begin{tabular}{@{}p{2.98cm}p{0.68cm}p{7.8cm}@{}}
\hline
\multicolumn{1}{c}{\textbf{Terminology}}        & \multicolumn{1}{c}{\textbf{Notation}} & \multicolumn{1}{c}{\textbf{Meaning}}                                                                                                                                                                                                                                                                                \\ \hline
\textbf{Scenario}                               & \multicolumn{1}{c}{$\xi_{\rm s}$}                         &  A possible outcome of $\xi$ that a probability is assigned to event $\xi=\xi_{\rm s}$, which can represent a set of data points.                                                                                                                                                                                 \\
\textbf{Data Point}                             & \multicolumn{1}{c}{$\xi_{\rm d}$}                         & The value of $\xi$ that is measured in history, which can be considered an independent and identically distributed (i.i.d.) sample.                                                                                                                                                                                    \\
\textbf{Probable scenario}                      & \multicolumn{1}{c}{--}                                    & A scenario of $\xi$ that satisfies $\mathbb{P}[\xi=\xi_{\rm s}] \ge \alpha$,  where $\alpha$ is a user-defined probability threshold.                                                                                                                                                                               \\
\textbf{Probable data point}                    & \multicolumn{1}{c}{--}                                    & A data point $\xi_{\rm d}$ of $\xi$ that satisfies $\xi_{\rm d}=\xi_{\rm s}$ is called a probable data point, where $\xi_{\rm s}$ is a probable scenario.                                                                                                                                                           \\
\textbf{Scenario set}                           & \multicolumn{1}{c}{$\mathcal{S}$}                         & $\{\xi^{(k)}_{\rm s},\; k=1,\ldots,S\}$: A set of scenarios where each of its elements is unique, i.e., $\xi^{(i)}_{\rm s} \neq \xi^{(j)}_{\rm s}$ if $i \neq j$ ($1 \le i,j\le S$).                                                                                                                                \\
\textbf{Data set}                               & \multicolumn{1}{c}{$\mathcal{D}$}                         & $\{\xi^{(k)}_{\rm d},\; k=1,\ldots,D\}$: a \textbackslash{}textbf\{finite multiset\} of data points of $\xi$, where there may be multiple instances for each of its elements, that is, $\xi^{(i)}_{\rm d} = \xi^{(j)}_{\rm d}$ is possible even if $i \neq j$ ($1 \le i,j\le S$). Note that $D$ is generally large. \\
\textbf{Particle representation of $\Upxi$}     & \multicolumn{1}{c}{$\mathcal{S}^\forall$}                 & A scenario set that contains all scenarios of $\xi$. $\mathcal{S}^\forall$ is not a multiset and generally \textbf{infinite} for continuous or mixed-integer $\xi$.                                                                                                                                      \\ \hline
\end{tabular}
\end{table}

\begin{figure}
    \centering
    \includegraphics[width=4.8cm]{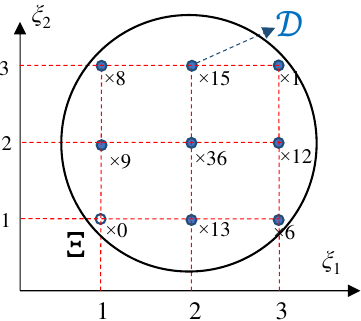}
%\vspace{-6pt}
\caption{An illustrative example of scenarios and data points.}
\label{fig:illu_examp1}

\end{figure}

\subsubsection{Formulation}
Suppose that a historical data set $\mathcal{D}=\{\xi^{(k)}_{\rm d},\; k=1,\ldots,D\}$ is available, where each data point is an independent and identically distributed (iid) sample of $\xi$. Let $\mathcal{D}_\alpha \subseteq \mathcal{D}$ denote the subset of probable data points (as defined in Table 2). We define the DEP approximation of PECO as
\begin{align} \label{DDA}
\text{\textbf{DEP}}:\quad\quad \min_{x\in \mathbb{R}^n} \quad & (\text{\ref{obj_PECO}})  \nonumber \\
 \mathrm{s.t.} \quad  &  g(x,\xi_{\rm d}^{(k)}) \le 0,\;   (\forall \xi_{\rm d}^{(k)}  \in \mathcal{D}_{\rm emb} \subseteq \mathfrak{U}[\mathcal{D}_\alpha]) 
\end{align} 
where $\mathcal{D}_{\rm emb}$ represents the embedded data set which is a subset of $\mathfrak{U}[\mathcal{D}_\alpha]$. Note that a data point differs from a scenario in the context of DESP. 

\subsubsection{Some properties}
Using ``DEP($\Box$)" to denote the data-embedded optimization model (\ref{DDA}) where $\mathcal{D}_{\rm emb}=\Box$, 
\begin{equation} \label{FS}
    \mathcal{X}_{\rm D}(\Box)=\{ x \in \mathbb{R}^n \mid g(x,\xi_{\rm d}^{(k)}) \le 0\;   (\forall \xi_{\rm d}^{(k)}  \in \Box)\}
\end{equation}
and $x^*_{\rm D}(\Box)$ to represent its feasible set in the $x$-space and optimal solution, respectively, the following proposition reveals some basic properties of the proposed DEP.
\begin{proposition}\label{pro:underlying}
    Considering an arbitrary data set $\mathcal{D}_1$,
 \begin{enumerate}
%\vspace{-6pt}
\itemsep0em 
    \item $\mathcal{X}_{\rm D}(\mathcal{D}_2) \supseteq \mathcal{X}_{\rm D}(\mathcal{D}_1)$ if $\mathcal{D}_2 \subseteq \mathcal{D}_1$;
    \item there exists at least one subset $\mathcal{D}_3$ of $\mathcal{D}_1$ that satisfies $\mathcal{X}_{\rm D}(\mathcal{D}_3)=\mathcal{X}_{\rm D}(\mathcal{D}_1)$;
    \item there exists at least one subset $\mathcal{D}_4$ of $\mathcal{D}_1$ that satisfies $x^*_{\rm D}(\mathcal{D}_4)=x^*_{\rm D}(\mathcal{D}_1)$;
    \item when $\mathcal{D}_1$ is a multiset, $\mathcal{X}_{\rm D}(\mathcal{D}_5)=\mathcal{X}_{\rm D}(\mathcal{D}_1)$ if $\mathcal{D}_5$ is the underlying set of $\mathcal{D}_1$, i.e., $\mathcal{D}_5=\mathfrak{U}[\mathcal{D}_1]$.
     %\vspace{-6pt}
\end{enumerate}   
\end{proposition}
\textit{Proof}: If $\mathcal{D}_2 \subseteq \mathcal{D}_1$, then 
\begin{align}
    \mathcal{X}_{\rm D}(\mathcal{D}_1)&=\left\{x \in \mathbb{R}^n \middle\vert \begin{array}{l}
        g(x,\xi_{\rm d}^{(i)}) \le 0,\;\forall \xi_{\rm d}^{(k)}\in \mathcal{D}_2 \\
        g(x,\xi_{\rm d}^{(j)}) \le 0,\;\forall \xi_{\rm d}^{(j)}\in \mathcal{D}_1 \setminus \mathcal{D}_2
    \end{array} \right\} \nonumber \\
    &=\{x \in \mathcal{X}_{\rm D}(\mathcal{D}_2) \mid g(x,\xi_{\rm d}^{(j)}) \le 0,\;\forall \xi_{\rm d}^{(j)}\in \mathcal{D}_1 \setminus \mathcal{D}_2 \} \nonumber
  \\
    & \subseteq \mathcal{X}_{\rm D}(\mathcal{D}_2). \nonumber
\end{align}
The existence of $\mathcal{D}_3$ and $\mathcal{D}_4$ follows directly by letting $\mathcal{D}_3=\mathcal{D}_4=\mathcal{D}_1$.

For the fourth statement, let $\mathcal{D}_5=\{\xi_{\rm d}^{(i)}\,(i=1,\ldots,I)\}$ be the underlying set of $\mathcal{D}_1$, and represent $\mathcal{D}_1$ as $\{\xi_{\rm d}^{(i,j)}\,(i=1,\ldots,I;\,j=1,\ldots,J_i)\}$, where $\xi_{\rm d}^{(i,1)}=\cdots=\xi_{\rm d}^{(i,J_i)}$ for each $i$. Then, DEP($\mathcal{D}_1$) can be written as
 \begin{subequations} \label{itm1}
\begin{align}
 \min_{x\in \mathbb{R}^n} \quad & (\text{\ref{obj_PECO}})  \\
 \mathrm{s.t.} \quad  &g(x,\xi^{(i,1)}_{\rm d}) \le 0 \label{constr1_itm1} \\
 &g(x,\xi^{(i,j)}_{\rm d}) \le 0,\;(j=2,\ldots,J_i) \label{constr2_itm1} 
\end{align} 
\end{subequations}
where $i =1,\ldots,I$. Since $\xi^{(i,1)}_{\rm d}=,\ldots,=\xi^{(i,J_i)}_{\rm d}$ ($\forall i =1,\ldots,I$), constraints in (\ref{constr2_itm1}) are redundant and can be removed without affecting the solution. The resulting problem is equivalent to DEP($\mathcal{D}_5$), implying $\mathcal{X}_{\rm D}(\mathcal{D}_5)=\mathcal{X}_{\rm D}(\mathcal{D}_1)$.
 
 \hfill $\Box$

\subsection{Relation to PECO}
The DEP model (\ref{DDA}) can be a deterministic approximation or even an equivalent of PECO depending on the embedded data set $\mathcal{D}_{\rm emb}$.
\subsubsection{Logical equivalence}
 Recalling $\mathcal{S}_\alpha^{\forall}$ which was defined in Table \ref{tab:terminologies}, we introduce the following definition. 
\begin{definition}[Deterministic equivalent of PECO]\label{def:equivalent}
    When $\mathcal{D}_{\rm emb}=\mathcal{S}_\alpha^{\forall}$, DEP is considered a deterministic equivalent of PECO, i.e., $\mathcal{X}_{\rm D}(\mathcal{S}_\alpha^{\forall})=\mathcal{X}_{\rm P}$.
\end{definition}
In other words, if $\mathcal{D}_{\rm emb}$ contains all probable scenarios (i.e., $\mathcal{D}_{\rm emb}=\mathcal{S}_\alpha^{\forall}$, the full set of realizations with probability $\ge \alpha$), then DEP is exactly equivalent to PECO (\ref{PECO2}). Nevertheless, DEP($\mathcal{S}_\alpha^{\forall}$) may involve infinitely many constraints when $\xi$ is continuous or mixed-integer, since $|\mathcal{S}_\alpha^{\forall}|=\infty$. To obtain a finite deterministic equivalent of PECO, we introduce the concepts of \emph{boundary-forming data set (BFDS)} and \emph{solution-determining data set (SDDS)}, from the perspectives of the feasible set and the optimal solution, respectively.

\subsubsection{Finite deterministic equivalent via the BFDS concept}

\begin{definition}[BFDS of the feasible space of DEP] \label{def:adpfs}
    A data set $\mathcal{B}^{\Box}=\{\xi_{\rm d}^{(k)},\,k=1,\ldots,B^{\Box} \}$ is called a BFDS of feasible set $\mathcal{X}_{\rm D}(\Box)$ if it is the SMALLEST subset of $\Box$ that satisfies:
    \begin{equation} \label{condition_adpfs}
        \mathcal{X}_{\rm D}(\mathcal{B}^{\Box}) = \{ x \in \mathbb{R}^n \mid g(x,\xi_{\rm d}^{(k)}) \le 0\;   (\forall \xi_{\rm d}^{(k)} \in \mathcal{B}^{\Box}) \}=\mathcal{X}_{\rm D}(\Box).
    \end{equation}
\end{definition}
 Following Definition \ref{def:adpfs}, let $\mathcal{B}^{\mathcal{S}_\alpha^{\forall}}$ denote the BFDS of $\mathcal{X}_{\rm D}(\mathcal{S}_\alpha^{\forall})$. For a general uncertain vector $\xi$ (i.e., $\xi$ can be integer, continuous, or mixed-integer), we have the following proposition. 
\begin{proposition}[Relations between the feasible set of DEP and PECO] \label{pro:adpfs}
    We have $\mathcal{X}_{\rm D} \supseteq \mathcal{X}_{\rm P}$ and, if and only if $\mathcal{B}^{\mathcal{S}_\alpha^{\forall}} \subseteq \mathcal{D}_{\rm emb}$, $\mathcal{X}_{\rm D}(\mathcal{D}_{\rm emb}) = \mathcal{X}_{\rm P}$.
\end{proposition}
Proof: Since $\mathcal{S}_\alpha^{\forall}$ is defined as the set that contains all scenarios of $\Upxi_\alpha$, we have $\mathfrak{U}[\mathcal{D}_\alpha] \subseteq \mathcal{S}_\alpha^{\forall}$. Thus, $\mathcal{X}_{\rm D}(\mathcal{D}_{\rm emb})\supseteq\mathcal{X}_{\rm D}(\mathfrak{U}[\mathcal{D}_\alpha]) \supseteq \mathcal{X}_{\rm D}(\mathcal{S}_\alpha^{\forall})$ where the inclusions follow from Part 1 of Propositions \ref{pro:underlying} and \ref{pro:feasisetpccsp}. By Definition \ref{def:equivalent}, $\mathcal{X}_{\rm D}(\mathcal{S}_\alpha^{\forall})=\mathcal{X}_{\rm P}$ which implies $\mathcal{X}_{\rm D}(\mathcal{D}_{\rm emb}) \supseteq \mathcal{X}_{\rm P}$. Condition $\mathcal{B}^{\mathcal{S}_\alpha^{\forall}} \subseteq \mathcal{D}_{\rm emb}$ indicates that $\mathcal{X}_{\rm D}(\mathcal{B}^{\mathcal{S}_\alpha^{\forall}}) \supseteq \mathcal{X}_{\rm D}(\mathcal{D}_{\rm emb})$. By Definition \ref{def:adpfs}, we have $\mathcal{X}_{\rm D}(\mathcal{B}^{\mathcal{S}_\alpha^{\forall}})=\mathcal{X}_{\rm D}(\mathcal{S}_\alpha^{\forall})=\mathcal{X}_{\rm P}$, which means $\mathcal{X}_{\rm P}\supseteq \mathcal{X}_{\rm D}(\mathcal{D}_{\rm emb})$. Then, we have $\mathcal{X}_{\rm D}(\mathcal{D}_{\rm emb}) = \mathcal{X}_{\rm P}$.

If $\mathcal{X}_{\rm D}(\mathcal{D}_{\rm emb}) = \mathcal{X}_{\rm P}$, we have $\mathcal{X}_{\rm D}(\mathcal{D}_{\rm emb}) =\mathcal{X}_{\rm D}(\mathcal{S}_\alpha^{\forall})$. Assuming $\mathcal{B}^{\mathcal{S}_\alpha^{\forall}} \not \subseteq \mathcal{D}_{\rm emb}$ and letting $\Tilde{\xi}_{\rm d} \not \in \mathcal{D}_{\rm emb}$ be an element in $\mathcal{B}^{\mathcal{S}_\alpha^{\forall}}$ and $\mathcal{D}_{\rm emb}^\prime=\mathcal{D}_{\rm emb} \bigcup\{ \Tilde{\xi}_{\rm d}\}$, there are two possible relations between $\mathcal{X}_{\rm D}(\mathcal{D}_{\rm emb})$ and $\mathcal{X}_{\rm D}(\mathcal{D}_{\rm emb}^\prime)$, i.e., $\mathcal{X}_{\rm D}(\mathcal{D}_{\rm emb})=\mathcal{X}_{\rm D}(\mathcal{D}_{\rm emb}^\prime)$ and $\mathcal{X}_{\rm D}(\mathcal{D}_{\rm emb})\supset\mathcal{X}_{\rm D}(\mathcal{D}_{\rm emb}^\prime)$. In the first relation, we have $\mathcal{X}_{\rm D}(\mathcal{D}_{\rm emb})=\mathcal{X}_{\rm D}(\mathcal{D}_{\rm emb}^\prime)=\mathcal{X}_{\rm D}(\mathcal{S}_\alpha^{\forall})$, which implies that removing $\Tilde{\xi}_{\rm d}$ from $\mathcal{B}^{\mathcal{S}_\alpha^{\forall}}$ (or $\mathcal{S}_\alpha^{\forall}$) does not affect the feasible set. This feature contradicts Definition \ref{def:adpfs} that $\mathcal{B}^{\mathcal{S}_\alpha^{\forall}}$ is the smallest subset of $\mathcal{S}_\alpha^{\forall}$ which satisfies condition (\ref{condition_adpfs}). Moreover, we know that $\mathcal{X}_{\rm D}(\mathcal{D}_{\rm emb}^\prime) \supseteq \mathcal{X}_{\rm D}(\mathcal{S}_\alpha^{\forall})$ since $\mathcal{D}_{\rm emb}^\prime \subseteq \mathcal{S}_\alpha^{\forall}$. The second relation contradicts the precondition of $\mathcal{X}_{\rm D}(\mathcal{D}_{\rm emb}) = \mathcal{X}_{\rm P}$. Hence, we have $\mathcal{B}^{\mathcal{S}_\alpha^{\forall}} \subseteq \mathcal{D}_{\rm emb}$ if $\mathcal{X}_{\rm D}(\mathcal{D}_{\rm emb}) = \mathcal{X}_{\rm P}$.

 \hfill $\Box$

  \begin{figure}[h!]
    \centering
    \includegraphics[width=0.5\linewidth]{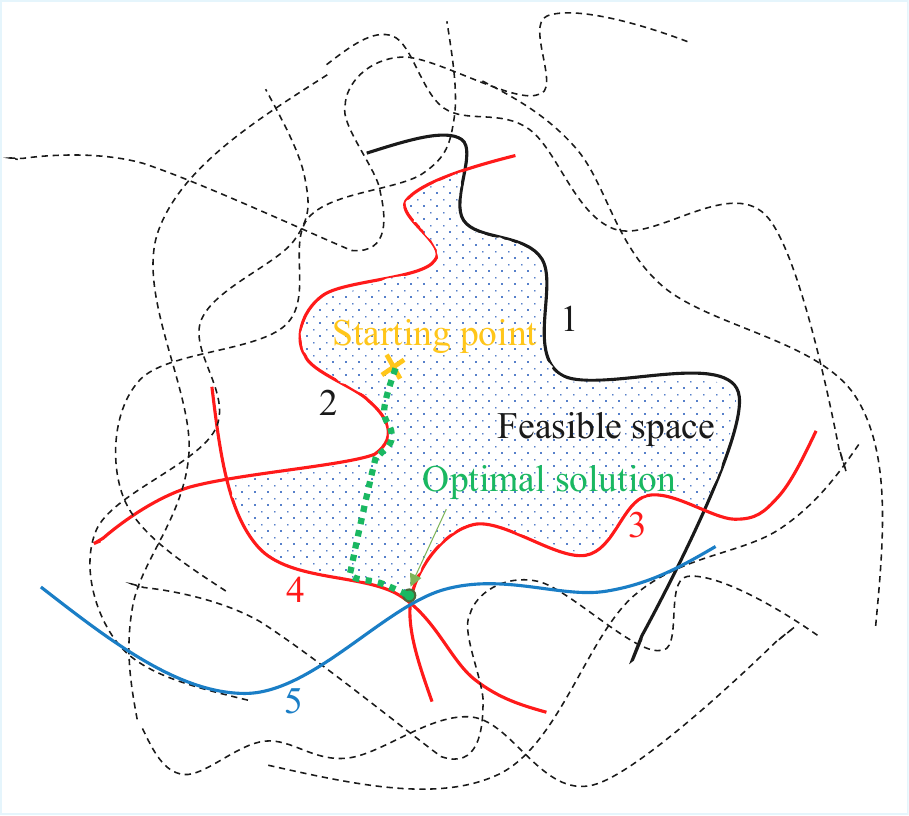}
    \caption{A pictorial interpretation of boundary-forming data points and the family of restricting data sets}
    \label{fig:activedata}
\end{figure}
  A pictorial interpretation of the BFDS concept is provided in Figure \ref{fig:activedata}, where constraints 1-4 are the boundary-forming ones. A boundary-forming data point is a data point that contributes at least one boundary-forming constraint. In other words, the BFDS of this case consists of the data points that contribute constraints 1-4. Proposition \ref{pro:underlying} guarantees the existence of $\mathcal{B}^{\Box}$ for DEP($\Box$). As illustrated in Figure \ref{fig:activedata}, the feasible space of DEP is typically determined by a limited number of boundary-forming data points, while the remaining data points are inactive and can be removed without affecting the feasible set. Proposition \ref{pro:adpfs} further implies that if the embedded data set $\mathcal{D}_{\rm emb}$ contains the whole BFDS of the deterministic equivalent DEP($\mathcal{S}_\alpha^{\forall}$), then the finite problem DEP($\mathcal{D}_{\rm emb}$) is equivalent to PECO. However, identifying the exact BFDS for a complex PECO problem is generally challenging. The next subsection introduces the concept of the SDDS family, based on which a finite deterministic approximation of PECO, along with an associated accuracy metric, can be developed.

\subsection{The concept of SDDS family}

\subsubsection{Definition and key properties of SDDS family}
The DEP model (\ref{DDA}) is essentially a nonlinear programming problem whose solution may be influenced by various factors, including but not limited to the numerical algorithms/solvers employed, the choice of initial solutions, the descent directions induced by the objective function, and the geometry of the feasible region. Since the choice of numerical algorithms/solvers and initial solutions is beyond the scope of this work, these factors are treated as fixed when defining the SDDS family. Let $x^*_{\rm D}(\Box)$ denote the optimal solution of DEP($\Box$). We then have the following definition.

\begin{definition}[SDDS family] \label{def:rdsf}
    Under a fixed numerical algorithm/solver and a given starting solution, a data set $\mathcal{R}_i^{\Box}=\{\xi^{(k)}_{\rm d},\; k=1,\ldots,R_i^{\Box}\} \subseteq \Box$ is called a solution-determining data set of DEP($\Box$) if:
    \begin{enumerate}
        \item $x^*_{\rm D}(\mathcal{R}_i^{\Box})=x^*_{\rm D}(\Box)$; and
        \item removing any data point from $\mathcal{R}_i^{\Box}$ changes the solution, i.e., $x^*_{\rm D}(\mathcal{R}_i^{\Box}\setminus\{\xi^{(k)}_{\rm d}\}) \neq x^*_{\rm D}(\Box)$ for all $k =1,\ldots,R_i^{\Box}$.
    \end{enumerate}
 The SDDS of DEP($\Box$) may not be unique and all such sets form the SDDS family $\mathfrak{R}^\Box=\{\mathcal{R}_i^{\Box}\,(i=1,\ldots,R) \}$. 
\end{definition}

SDDS families generally possess the following key properties:
\begin{itemize}
    \item The SDDS family is defined under a fixed numerical algorithm/solver and starting solution, since different choices may lead to different solutions $x^*_{\rm D}(\Box)$ and hence different $\mathfrak{R}^\Box$.
    \item An element of an SDDS (i.e., a solution-determining data point) may correspond to an active data point at $x^*_{\rm D}(\Box)$, where an active data point contributes at least one active constraint at the optimal solution. In the example shown in Figure \ref{fig:activedata}, the data points corresponding to constraints 3, 4, and 5 are solution-determining.
    \item An SDDS element may also correspond to a data point whose associated constraints influence the search trajectory of the numerical solver, e.g., constraint 2 in Figure \ref{fig:activedata}.
    \item The SDDS family $\mathfrak{R}^\Box$ may be empty. This occurs, for example, when none of the inequality constraints is binding.
    \item Let $\mathcal{R}^{\Box}_\prime$ be a subset of an SDDS $\mathcal{R}^{\Box}$. By Definition \ref{def:rdsf}, the only subset satisfying $x^*_{\rm D}(\mathcal{R}^{\Box}_\prime)=x^*_{\rm D}(\mathcal{R}^{\Box})$ is $\mathcal{R}^{\Box}$ itself. That is, all elements in an SDDS are jointly necessary to determine the optimal solution.
    \item If multiple SDDSs exist, then $\mathcal{R}_i^{\Box}\neq \mathcal{R}_j^{\Box}$ for $i \neq j$, and their intersection may be nonempty. In the example in Figure \ref{fig:activedata}, $\mathcal{R}_1=\{\xi^{(2)}_{\rm d},\,\xi^{(3)}_{\rm d},\,\xi^{(4)}_{\rm d}\}$ and $\mathcal{R}_2=\{\xi^{(2)}_{\rm d},\,\xi^{(4)}_{\rm d},\,\xi^{(5)}_{\rm d}\}$ form $\mathfrak{R}=\{\mathcal{R}_1,\, \mathcal{R}_2\}$, where $\xi^{(i)}_{\rm d}$ represents the data point that contributes to constraint $i$.
\end{itemize}

While BFDS is defined from the perspective of the feasible set, the SDDS family is defined from that of the optimal solution. Similar to BFDS, identifying the SDDS family a priori is generally difficult. In the remainder of this paper, we aim to make the SDDS concept more practical for solving DEP approximations of PECO by developing solution methods that rely on weaker assumptions. In the next subsection, we derive a finite DEP approximation with an accuracy metric based on an estimate of the SDDS cardinalities.

\subsubsection{Finite DEP approximation with accuracy metric via SDDS family}

Recall that DEP($\mathcal{S}_\alpha^{\forall}$) is a deterministic equivalent of PECO, whereas DEP($\mathcal{D}_{\rm emb}$) is a finite deterministic approximation. The former is generally an infinite problem, while the latter is finite. Suppose $\mathcal{D}_{\rm emb}$ consists of $z$ data points randomly selected from $\mathcal{S}_\alpha^{\forall}$, this subsection studies the relationship between $x^*_{\rm D}(\mathcal{D}_{\rm emb})$ and $x^*_{\rm D}(\mathcal{S}_\alpha^{\forall})$. Let $\mathfrak{R}^{\mathcal{S}_\alpha^{\forall}}=\{\mathcal{R}_1^{\mathcal{S}_\alpha^{\forall}},\ldots,\mathcal{R}_R^{\mathcal{S}_\alpha^{\forall}} \}$ denote the SDDS family of DEP($\mathcal{S}_\alpha^{\forall}$), we impose the following condition.

\textbf{Condition 1}. All $x^*_{\rm D}$'s are obtained by the same algorithm/solver and the same starting point.
\begin{proposition} \label{pro:adpos2}
    Under condition 1, $x^*_{\rm D}(\mathcal{D}_{\rm emb})=x^*_{\rm P}$ if and only if $\mathcal{D}_{\rm emb} \supseteq \mathcal{R}_i^{\mathcal{S}_\alpha^{\forall}}$ ($\exists i \in \{1,\dots,R \}$).
\end{proposition}

\textit{Proof}: Given the condition $\mathcal{D}_{\rm emb} \supseteq \mathcal{R}_i^{\mathcal{S}_\alpha^{\forall}}$ ($\exists i \in \{1,\dots,R \}$), we assume that $\mathcal{D}_{\rm emb} \supseteq \mathcal{R}_j^{\mathcal{S}_\alpha^{\forall}}$. Then, we have $x^*_{\rm D}(\mathcal{R}_j^{\mathcal{S}_\alpha^{\forall}})=x^*_{\rm D}(\mathcal{S}_\alpha^{\forall})=x^*_{\rm P}$ by Definitions \ref{def:equivalent} and \ref{def:rdsf}. Thus, removing constraints corresponding to ($\mathcal{S}_\alpha^{\forall} \setminus \mathcal{R}_j^{\mathcal{S}_\alpha^{\forall}}$) from DEP($\mathcal{S}_\alpha^{\forall}$) does not change the optimal solution until the problem reduces to DEP($\mathcal{R}_j^{\mathcal{S}_\alpha^{\forall}}$) under condition 1. Given that $\mathcal{S}_\alpha^{\forall} \supseteq \mathcal{D}_{\rm emb}$ according to the logical meaning of $\mathcal{S}_\alpha^{\forall}$, we have $(\mathcal{S}_\alpha^{\forall} \setminus \mathcal{R}_j^{\mathcal{S}_\alpha^{\forall}}) \supseteq (\mathcal{S}_\alpha^{\forall} \setminus \mathcal{D}_{\rm emb})$. Hence, if DEP($\mathcal{S}_\alpha^{\forall}$) reduces to DEP($\mathcal{D}_{\rm emb}$), its optimal solution would not change, i.e., $x^*_{\rm D}(\mathcal{D}_{\rm emb})=x^*_{\rm D}(\mathcal{S}_\alpha^{\forall})=x^*_{\rm P}$. 

Under condition 1, if $x^*_{\rm D}(\mathcal{D}_{\rm emb})=x^*_{\rm P}$, we have $x^*_{\rm D}(\mathcal{D}_{\rm emb})=x^*_{\rm D}(\mathcal{S}_\alpha^{\forall})$, which implies that removing constraints corresponding to $(\mathcal{S}_\alpha^{\forall} \setminus \mathcal{D}_{\rm emb})$ from DEP($\mathcal{S}_\alpha^{\forall}$) does not change the solution. According to part 3 of  Proposition \ref{pro:underlying}, there exists a minimal subset $\mathcal{D}^{\min}_{\rm emb} \subseteq \mathcal{D}_{\rm emb}$ such that removing any element changes the solution. Hence, $\mathcal{D}^{\min}_{\rm emb}$ is an SDDS, implying $\mathcal{D}_{\rm emb}$ contains at least one SDDS.

\hfill $\Box$

Proposition \ref{pro:adpos2} implies that solving DEP($\mathcal{D}_{\rm emb}$)--a finite problem--can obtain the exact solution of PECO if $\mathcal{D}_{\rm emb}$ contains at least one SDDS of DEP($\mathcal{S}_\alpha^{\forall}$). However, since the SDDS family is unknown in advance, this condition is difficult to verify. Although it is hard to know which data points are the solution-determining ones, if we know the size of each subset of the SDDS family, we can use Theorem \ref{thm:varrho} to determine the needed $z$, i.e., the size of $\mathcal{D}_{\rm emb}$, under assumption 1.
\smallskip

\noindent
\textbf{Assumption 1}.The data set $\mathcal{D}$ is sufficiently large such that
\[
\mathfrak{U}[\mathcal{D}_\alpha] \supseteq \mathcal{R}_i^{\mathcal{S}_\alpha^{\forall}}, \quad \forall i \in \{1,\dots,R\}.
\]

Let $\varrho=\mathbb{P}[x^*_{\rm D}(\mathcal{D}_{\rm emb})=x^*_{\rm P}]$. Define subsets $\mathfrak{R}^{\mathcal{S}_\alpha^{\forall}}_j=\{ \mathcal{R}_i^{\mathcal{S}_\alpha^{\forall}},\,i \in \mathcal{J}_j\}$ for $j \in \mathcal{J}$, where $|\mathcal{J}|=2^R-1$. Let $\underline{D}_\alpha=|\mathfrak{U}[\mathcal{D}_\alpha]|$.

\begin{theorem}\label{thm:varrho}
    Under Assumption 1 and Condition 1, the probability of $x^*_{\rm D}(\mathcal{D}_{\rm emb})=x^*_{\rm P}$ is 
    \begin{equation} \label{varrho}
       \varrho(z)= \frac{1}{\binom{\underline{D}_\alpha}{z}}\sum_{j \in \mathcal{J}} (-1)^{J_j+1}\binom{\underline{D}_\alpha - \bar R_j^{\mathcal{S}_\alpha^{\forall}}}{z- \bar R_j^{\mathcal{S}_\alpha^{\forall}}} ,
    \end{equation}
  where $J_j=|\mathcal{J}_j|$, $\bar R_j^{\mathcal{S}_\alpha^{\forall}}=\left|\bigcup_{i \in \mathcal{J}_j} \mathcal{R}_i^{\mathcal{S}_\alpha^{\forall}}\right|$, and $z\ge\bar R_j^{\mathcal{S}_\alpha^{\forall}}$ ($\forall j \in \mathcal{J}$). 
\end{theorem}
 \textit{Proof}: Proposition \ref{pro:adpos2} indicates that $\varrho=\mathbb{P}[x^*_{\rm D}(\mathcal{D}_{\rm emb})=x^*_{\rm P}]=\mathbb{P}[\mathcal{D}_{\rm emb} \supseteq \mathcal{R}_i^{\mathcal{S}_\alpha^{\forall}}\, (\exists i \in \{1,\dots,R \})]$ under condition 1. To evaluate the probability of event $\mathcal{D}_{\rm emb} \supseteq \mathcal{R}_i^{\mathcal{S}_\alpha^{\forall}}$ ($\exists i \in \{1,\dots,R \}$), we define the following events:
 \begin{itemize}
    \item E$_i$ ($ i=1,\ldots,R$): when $z$ data points are randomly selected from $\mathfrak{U}[\mathcal{D}_\alpha]$, they contain all data points in $\mathcal{R}_i^{\mathcal{S}_\alpha^{\forall}}$.
    \item E$^*$: when $z$ data points are randomly selected from $\mathfrak{U}[\mathcal{D}_\alpha]$, they contain all data points in at least one of $\mathcal{R}_i^{\mathcal{S}_\alpha^{\forall}}$ ($ i=1,\ldots,R$), i.e., E$^*$: $\mathcal{D}_{\rm emb} \supseteq \mathcal{R}_i^{\mathcal{S}_\alpha^{\forall}}\, (\exists i \in \{1,\dots,R \})$.
    \item A$_j$ ($j \in \mathcal{J}$):  when $z$ data points are randomly selected from $\mathfrak{U}[\mathcal{D}_\alpha]$, they contain all data points in the $j$th subset of the SDDS family $\mathfrak{R}^{\mathcal{S}_\alpha^{\forall}}$, i.e., $\mathfrak{R}^{\mathcal{S}_\alpha^{\forall}}_j$, where $\mathcal{J}$ is the index set of all subsets of $\mathfrak{R}^{\mathcal{S}_\alpha^{\forall}}$.
\end{itemize}
It suffices to know that E$^*$=$\bigcup_{i=1}^R$E$_i$. By inclusion-exclusion, we have 
\begin{equation} \label{thm6_proof_1}
    \mathbb{P}[\text{E}^*]=\mathbb{P}\left[\bigcup_{i=1}^R\text{E}_i \right]=\sum_{j \in \mathcal{J}} (-1)^{J_j+1}\mathbb{P}\left[\bigcap_{i \in \mathcal{J}_j}\text{E}_i \right].
\end{equation}
It's not hard to know that A$_j=\bigcap_{i \in \mathcal{J}_j}$E$_i$. Let $\bar R_j^{\mathcal{S}_\alpha^{\forall}}$ denote the number of data points in $\mathfrak{R}^{\mathcal{S}_\alpha^{\forall}}_j$, we have
\begin{equation} \label{thm6_proof_2}
    \mathbb{P}\left[\bigcap_{i \in \mathcal{J}_j}\text{E}_i \right]=\mathbb{P}[\text{A}_j]=\frac{\binom{\underline{D}_\alpha - \bar R_j^{\mathcal{S}_\alpha^{\forall}}}{z- \bar R_j^{\mathcal{S}_\alpha^{\forall}}}}{\binom{\underline{D}_\alpha}{z}}
\end{equation}
 where $\underline{D}_\alpha=|\mathfrak{U}[\mathcal{D}_\alpha]|$. Combining (\ref{thm6_proof_1}) and (\ref{thm6_proof_2}), we have (\ref{varrho}).

\hfill $\Box$

\section{The Data-embedded Solution Paradigm} \label{sec5}

Using DEP as the deterministic approximation, a Data-embedded Solution Paradigm (DESP) is proposed for solving PECO problems. The overall procedure is summarized as follows: 
\begin{enumerate}
    \item Obtain $\mathcal{D}_\alpha$ from the original data set $\mathcal{D}$;
    \item Learn appropriate values of  $\bar R^{\mathcal{S}_\alpha^{\forall}}=[\bar R_1^{\mathcal{S}_\alpha^{\forall}},\ldots,\bar R_J^{\mathcal{S}_\alpha^{\forall}}]$ for a given problem from historical data;
    \item construct $\mathcal{D}_{\rm emb}$ by randomly selecting $z$ data points from $\mathfrak{U}[\mathcal{D}_\alpha]$, where $z$ is determined from (\ref{varrho}) based on a user-specified $\varrho$;
    \item Solve DEP($\mathcal{D}_{\rm emb}$) using off-the-shall solvers.
\end{enumerate}
While the theoretical foundation of Step 3 has been established in Section 3, this section provides further details on Steps 1 and 2.

\subsection{Obtaining $\mathcal{D}_\alpha$}
This subsection presents a practical approach of obtaining $\mathcal{D}_\alpha$ from $\mathcal{D}$ based on the following simple definition of the joint probability of data points. 
\begin{definition}
Let $\mathcal{D}_j^\eta$ denote the set of data points within the $\eta$-vicinity of a data point $\xi^{(j)}_{\rm d}$ in $\mathcal{D}$, i.e.,
\[
\mathcal{D}_j^\eta=\{\xi^{(k)}_{\rm d} \in \mathcal{D} \mid \|\xi^{(k)}_{\rm d}-\xi^{(j)}_{\rm d}\| \le \eta \},
\]
and let $D_j^\eta=|\mathcal{D}_j^\eta|$. The joint probability is then defined as
\[
\mathbb{P}[\xi=\xi^{(j)}_{\rm d}]=D_j^\eta/D,
\]
where $\eta$ is a small positive scalar.
\end{definition}
A pictorial illustration of the above definition is given in Figure \ref{fig:d32}. The choice of bandwidth $\eta$ affects the accuracy of the estimated probability $\mathbb{P}[\xi=\xi^{(j)}_{\rm d}]$, and its optimal value varies across applications. Fortunately, established methods such as plug-in estimators and cross-validation can be used to select an appropriate $\eta$. Based on this definition, Algorithm \ref{alg:alpha} outlines a procedure for constructing $\mathcal{D}_\alpha$ from $\mathcal{D}$.
\begin{figure}[h]
\centering
\includegraphics[width=8.8cm]{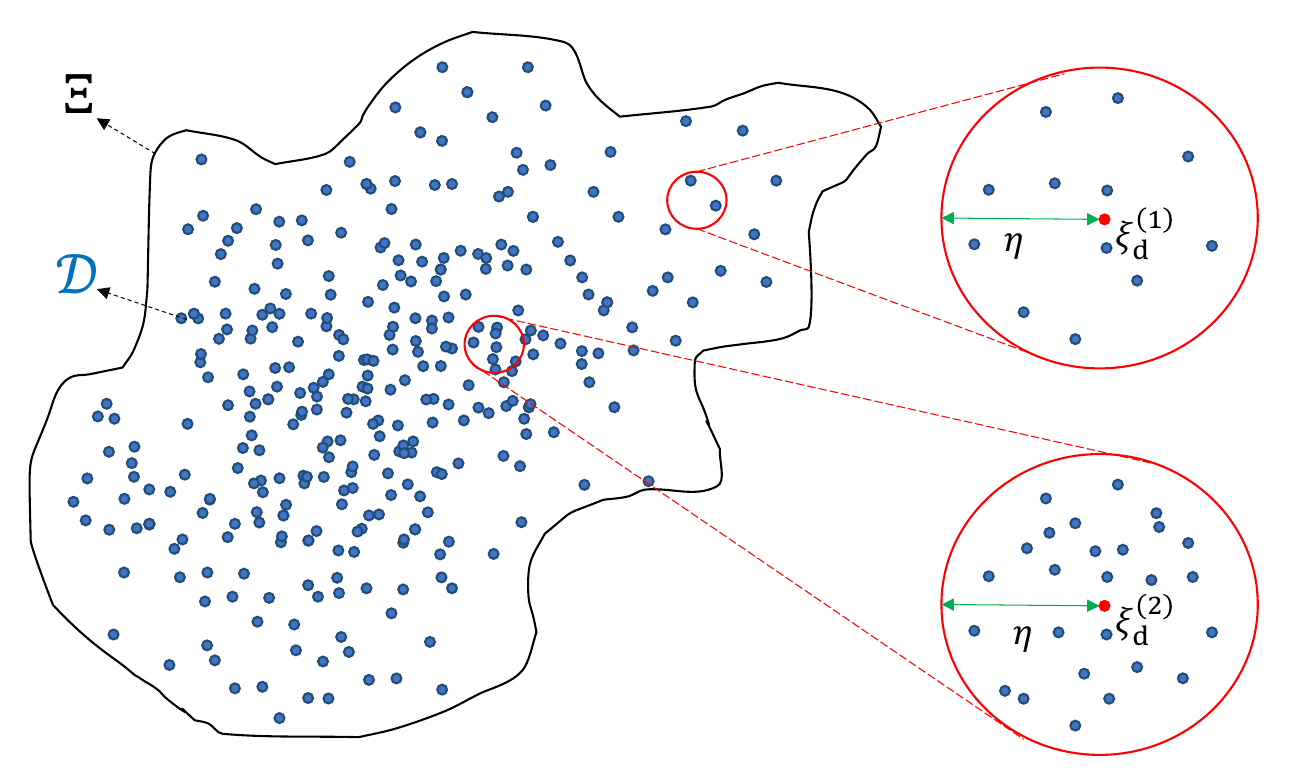}
\caption{Pictorial illustration on the joint probability of a data point, where the dots represent the data points (namely historical realizations of $\xi$ which are assumed to be i.i.d.) and the area enclosed by the solid curve is the support set of $\xi$. In this example, $\mathbb{P}[\xi^{(1)}_{\rm d}]<\mathbb{P}[\xi^{(2)}_{\rm d}]$ since $D_1^\eta<D_2^\eta$.}
\label{fig:d32}
\vspace{-6pt}
\end{figure}

\begin{algorithm} \label{algorithm}
\caption{: Procedure for constructing $\mathcal{D}_\alpha$}
\label{alg:alpha}
\begin{algorithmic}
 \State Given a data set $\mathcal{D}$ ($D=|\mathcal{D}|$) and set $i$=1:
 \State 1. Count the number of data points in the $\eta$-vicinity of $\xi^{(i)}_{\rm d}$ and save it to set $\mathcal{D}_\alpha$ if $D_i^\eta \ge \alpha D$ where $\eta$ is the selected bandwidth;
 \State  2. stop if $i=D$, otherwise set $i=i+1$ and repeat step 1. 
\end{algorithmic}
\end{algorithm}

\subsection{Learning the values of $\bar R^{\mathcal{S}_\alpha^{\forall}}$}
 \subsubsection{A simple option}
 While $\xi$ represents uncertain parameters, we use $\delta$ to denote deterministic parameters that also influence the problem. The optimal solution $x^*$ depends on both $\xi$ and $\delta$. Furthermore, each $x^*$ is associated with an SDDS family $\mathfrak{R}^{\mathcal{S}_\alpha^{\forall}}$. Therefore, under Condition 1, there exists a one-to-one correspondence between $\delta$ and $\bar R^{\mathcal{S}_\alpha^{\forall}}$. Suppose the problem has been solved $K$ times historically. This yields a dataset 
\[
\{(\bar R^{\mathcal{S}_\alpha^{\forall}}_{(k)},\delta_{(k)}),\; k=1,\ldots,K\},
\]
which can be used to learn a mapping $\bar R^{\mathcal{S}_\alpha^{\forall}}=\mathcal{L}(\delta)$. For a new instance with parameter $\delta$, this learned function provides an estimate of $\bar R^{\mathcal{S}_\alpha^{\forall}}$, which can then be used in (\ref{varrho}) to determine the required $z$ for a desired confidence level $\varrho$. 

 \subsubsection{Rationality of Theorem \ref{thm:varrho} in a machine learning setting}
If the embedded data set $\mathcal{D}_{\rm emb}$ is formed by randomly selecting $z$ data points from $\mathfrak{U}[\mathcal{D}_\alpha]$, Theorem \ref{thm:varrho} provides a probabilistic measure of the accuracy of approximating PECO using DEP. Although the theorem does not require full knowledge of the SDDS family $\mathfrak{R}^{\mathcal{S}_\alpha^{\forall}}$, it depends on the sizes of its subsets, i.e., $\bar R_j^{\mathcal{S}_\alpha^{\forall}}$ ($j \in \mathcal{J}$), which are also not directly observable prior to solving DEP($\mathfrak{U}[\mathcal{D}_\alpha]$).

However, compared to identifying the full SDDS family, predicting the sizes of its subsets is significantly easier when historical solution data are available. These quantities can be estimated using standard machine learning methods. To illustrate, consider an analogy: in a room with several people, predicting $\bar R_j^{\mathcal{S}_\alpha^{\forall}}$ corresponds to estimating the number of people present, whereas identifying the full SDDS family corresponds to determining both the number of people and their individual identities. The former is substantially simpler than the latter, which justifies the practical applicability of Theorem \ref{thm:varrho}.

 %$n$ sets $A_1$, $\ldots$, $A_n \subset B$ satisfy that $A_i \nsubseteq A_j$ if $i \neq j$ ($\forall i,j=1,\ldots,n$). Let's consider an experiment of randomly selecting $M$ elements from $B$ and let $P_i$ denote the probability that the $M$ selected elements contain all elements in $A_1$ ($i=1,\ldots,n$). Further let $P$ denote the probability that the $M$ selected elements contain all elements in $\bigcup_{i=1}^n A_i$, what's $P$? Note that $M$ is bigger than the number of elements in $\bigcup_{i=1}^n A_i$. If $C=A_i\bigcap A_j$, do we have $P_C=P_iP_j$?
 
 \section{Conclusion and Future Work}
This paper proposed a novel optimization paradigm, termed Probable-Event Constrained Optimization (PECO), to address decision-making problems under uncertainty when the underlying probability distribution is unknown. By enforcing feasibility across all realizations with sufficiently high probability, PECO provides a modeling framework that is both practically meaningful and aligned with engineering requirements. 

To overcome the computational challenges associated with PECO, we developed a Data-embedded Solution Paradigm (DESP), in which a data-embedded programming (DEP) model serves as a deterministic approximation. Theoretical analysis established the relationship between DEP and PECO from both feasible set and optimal solution perspectives. In particular, the concepts of boundary-forming data sets (BFDS) and solution-determining data set (SDDS) families were introduced to characterize when a finite DEP formulation can exactly recover the PECO solution. Building on these insights, we derived a probabilistic accuracy guarantee for DEP based on random data embedding. The resulting expression quantifies the likelihood that a finite DEP solution matches the true PECO optimum as a function of the embedded data size. This provides a principled mechanism for balancing computational tractability and solution accuracy. Furthermore, we outlined a practical workflow for implementing DESP, including procedures for identifying probable data points and learning key structural parameters from historical data.

Overall, the proposed framework offers a new perspective on optimization under uncertainty by integrating data-driven modeling with rigorous theoretical guarantees. It bridges the gap between purely distribution-based approaches and data-centric methods, and opens up new opportunities for scalable and reliable decision-making in complex systems. While the proposed Data-embedded Solution Paradigm (DESP) provides a principled and practical framework for solving PECO problems, several directions remain for further improvement and extension.

The current approach for constructing $\mathcal{D}_\alpha$ relies on a simple neighborhood-based estimation of joint probabilities using a fixed bandwidth parameter $\eta$. Although effective and easy to implement, this method may suffer from sensitivity to the choice of $\eta$, especially in high-dimensional settings or when the data distribution exhibits heterogeneity. Future work will focus on developing more advanced and adaptive methods for identifying $\mathcal{D}_\alpha$. One promising direction is to leverage techniques from density estimation, such as kernel density estimation with data-driven bandwidth selection, to improve the accuracy of probability estimation. In addition, clustering-based methods or manifold learning techniques could be explored to better capture the intrinsic structure of the data and identify high-probability regions more robustly. Another important direction is to consider statistical learning approaches that directly classify data points into probable and non-probable categories, potentially avoiding explicit probability estimation. Such methods may improve scalability and robustness, particularly for large-scale and high-dimensional data sets.

The current framework of learning $\bar R^{\mathcal{S}_\alpha^{\forall}}$ assumes that the values of $\bar R^{\mathcal{S}_\alpha^{\forall}}$ can be learned from historical data via a mapping $\bar R^{\mathcal{S}_\alpha^{\forall}}=\mathcal{L}(\delta)$. While this provides a simple and practical solution, the accuracy of the learned model is critical to the overall performance of DESP. Future research will investigate more sophisticated learning strategies for estimating $\bar R^{\mathcal{S}_\alpha^{\forall}}$. In particular, machine learning models with stronger generalization capabilities, such as ensemble methods or deep neural networks, may be employed to capture complex relationships between problem parameters $\delta$ and the structure of the SDDS family. Moreover, incorporating uncertainty quantification into the learning process is an important direction. Instead of producing point estimates, probabilistic predictions of $\bar R^{\mathcal{S}_\alpha^{\forall}}$ could be used to derive more robust choices of $z$ in (\ref{varrho}), thereby improving the reliability of the resulting DEP solutions. Finally, it would be valuable to explore theoretical guarantees on the learnability of $\bar R^{\mathcal{S}_\alpha^{\forall}}$, including sample complexity and generalization bounds. Such analysis could provide deeper insights into the conditions under which DESP achieves reliable performance.

Beyond the two core components above, several broader extensions are worth exploring. These include developing adaptive sampling strategies for constructing $\mathcal{D}_{\rm emb}$, integrating DESP with online or streaming data settings, and extending the framework to handle dynamic or time-dependent uncertainty. Investigating these directions would further enhance the applicability and robustness of the proposed paradigm.

\bmhead{Declaration}
This version is excerpted from a full length paper that is currently under review.

\bibliography{references}

\end{document}